\documentclass[12pt,reqno]{amsart}
\newtheorem{thm}{Theorem}[section]
\newtheorem{defi}{Definition}[section]
\newtheorem{lem}{Lemma}[section]
\newtheorem{cor}{Corollary}[section]
\newtheorem{prop}{Proposition}[section]
\theoremstyle{definition}
\newtheorem{rem}{Remark}[section]
\newtheorem{exmp}{Example}[section]
\newtheorem{note}{Note}[section]
\newcommand{\be}{\begin{equation}}
\newcommand{\ee}{\end{equation}}
\newcommand{\bea}{\begin{eqnarray}}
\newcommand{\eea}{\end{eqnarray}}
\newcommand{\beb}{\begin{eqnarray*}}
\newcommand{\eeb}{\end{eqnarray*}}
\usepackage{amssymb,amsfonts,amsthm,setspace,indentfirst,mathrsfs}
\usepackage{minibox}
\usepackage{enumitem}
\usepackage{tikz-cd}
\usepackage{centernot}
\usepackage{stmaryrd}
\usepackage{hyperref}
\numberwithin{equation}{section}

\begin{document}
\title[On $\mathcal{I}$-convergence of sequences of functions and uniform conjugacy]{On $\mathcal{I}$-convergence of sequences of functions and uniform conjugacy}

\author[A. K. Banerjee and N. Hossain]{$^1$Amar Kumar Banerjee and $^2$Nesar Hossain}

\address{$^{1,2}$Department of Mathematics, The University of Burdwan, Golapbag, Burdwan - 713104, West Bengal, India.}

\email{$^1$akbanerjee1971@gmail.com ;akbanerjee@math.buruniv.ac.in}
\email{$^2$nesarhossain24@gmail.com}


\maketitle

\begin{abstract}
 In this paper we introduce the notion of $\mathcal{I^*}\text{-}\alpha$-uniform equal convergence  and $\mathcal{I^*}\text{-}\alpha$-strong uniform equal convergence of sequences of functions and then investigate some lattice properties of $\Phi ^{\mathcal{I^*}\text{-}\alpha\text{-}u.e.}$ and $\Phi ^{\mathcal{I^*}\text{-}\alpha \text{-}s.u.e.}$, the classes of all functions which are $\mathcal{I^*}\text{-}\alpha$-uniform equal limits and $\mathcal{I^*}\text{-}\alpha$-strong uniform equal limits of sequences of functions respectively obtained from a class of functions $\Phi$.  We have also shown that $\mathcal{I}$-exhaustiveness, $\mathcal{I}$-uniform and  $\mathcal{I}\text{-}\alpha$-  convergence of sequences of functions are preserved under uniform conjugacy.
\end{abstract}

\section{\textbf{Introduction}}
The two kinds of generalizations of statistical convergence \cite{Fast, Steinhaus} were introduced by Kostyrko et al. \cite{Kostyrko Salat Wilczynski} which was named as $\mathcal{I}$ and $\mathcal{I^*}$-convergence based on the structure of ideals of the set $\mathbb{N}$ of positive integers. As a natural consequence over the years, some significant investigations on this convergence were studied in many directions in different spaces \cite{Banerjee Banerjee 2016, Banerjee Banerjee 2018, Banerjee Paul, Das Kostyrko Wilczynski Malik, Lahiri Das}. 
Besides uniform convergence and pointwise convergence,  different types of converegence of sequences of real valued functions were studied and their significant properties were  developed analogously. For example the ideas of discrete, equal \cite{Csaszar Laczkovich 1975, Csaszar Laczkovich 1979} (also known as quasinormal convergence \cite{Bukovska}), uniform equal, uniform discrete \cite{Papanastassiou 2002}, $\alpha$-convergence (in the literature, also known as continuous convergence \cite{Stoilov}), $\alpha$-equal, $\alpha$-uniform equal, $\alpha$-strong uniform equal convergence \cite{Das Papanastassiou} were introduced by several authors. Later,  the idea of   $\mathcal{I^*}$-uniform discrete, $\mathcal{I^*}$-uniform equal, $\mathcal{I^*}$-strong uniform equal, $\mathcal{I}$-strong equal convergence of sequences of real valued functions were studied in \cite{Das Dutta, Das Dutta Pal, Sengupta}

In this paper we investigate some lattice properties of the classes of all functions which are $\mathcal{I^*}\text{-}\alpha$-uniform equal limits and $\mathcal{I^*}\text{-}\alpha$-strong uniform equal limits of sequences of functions following the investigations of \cite{Das Das}. Tian and Chen (\cite{Tian  Chen}) introduced the notion of uniform conjugacy for sequences of maps.  We  show that different types of $\mathcal{I}$- convergence of sequences of functions are preserved under uniform conjugacy.

\section{\textbf{Preliminaries}}
Throughout the paper $\mathbb{N}$ and $\mathbb{R}$  denote the set of all positive integers and the set of all real numbers respectively. Now we recall some basic definitions and notations. 
\begin{defi}
A family $\mathcal{I}\subset 2^Y$ of subsets of a non empty set $Y$ is said to be an ideal if the following conditions hold.\\
$(i)$ $A, B\in \mathcal{I}\Rightarrow A\cup B \in \mathcal{I}$, \\
$(ii)$ $A\in \mathcal{I}, B\subset A\Rightarrow B\in \mathcal{I}$.
\end{defi} 

From the definition it follows that $\phi \in \mathcal{I}$. $\mathcal{I}$ is called non trivial if $Y\notin \mathcal{I}$ and proper if $\mathcal{I}\neq \{\phi \}$. An ideal $\mathcal{I}$ in $Y$ is said to be an admissible ideal if $\{x\}\in \mathcal{I}$ for each $x\in Y$.
Let $\mathcal{I}$ is non trivial proper ideal in $Y$ then the family of sets $\mathcal F(\mathcal{I})=\{Y\setminus A : A \in \mathcal{I}\} $ is a filter on $Y$ which  is called the filter associated with the ideal $\mathcal{I}$. Throughout the paper $\mathcal{I}$ will stand for an admissible ideal of $\mathbb{N}$. 
A sequence  $\{x_n \}_{n\in \mathbb{N}}$ of real numbers is said to be $\mathcal{I}$-convergent to $x\in \mathbb{R}$ if for each $\varepsilon > 0$ the set $A(\varepsilon)=\{n\in \mathbb{N}: |x_n - x|\geq \varepsilon \}\in \mathcal{I}$ \cite{Kostyrko Salat Wilczynski}.
The sequence  $\{x_n \}_{n\in \mathbb{N}}$ of real numbers is said to be $\mathcal{I^*}$-convergent to $x\in \mathbb{R}$ if there exists a set $M\in \mathcal F(\mathcal{I}), M =\{m_1 < m_2 < \cdots <m_k < \cdots \}$ such that $\lim _{k\to \infty} x_{m_k}=x$ \cite{Kostyrko Salat Wilczynski}.

We denote by $|A|$ to stand for the cardinality of the set $A$. Let $X$ be a non empty set. By a function on $X$, we mean a real valued function on $X$. Let $\Phi$ be an arbitrary class of functions defined on $X$. Then we have the following definitions.

 \begin{defi}\cite{Csaszar Laczkovich 1979}
$(a)$ $\Phi$ is called a lattice if $\Phi$ contains all constant functions and $f, g\in \Phi$ implies $max(f, g)\in \Phi$ and $min(f, g)\in \Phi.$ \\
$(b)$  $\Phi$ is called a translation lattice if it is a lattice and $f\in \Phi , c\in \mathbb{R}$ implies $f+c\in \Phi.$\\
$(c)$ $\Phi$ is called a congruence  lattice if it is a translation lattice and $f\in \Phi$ implies $-f\in \Phi.$\\
$(d)$ $\Phi$ is a weakly affine lattice if it is a congruence lattice and there is a set $C\subset (0, \infty)$ such that $C$ is not bounded and $f\in \Phi, c\in C$ implies $cf\in \Phi.$ \\
$(e)$ $\Phi$ is called an affine lattice if it is a congruence lattice and $f\in \Phi, c\in \mathbb{R}$ implies $cf\in \Phi.$ \\
$(f)$ $\Phi$ is called a subtractive lattice if it is a lattice and $f, g\in \Phi$ implies $f-g\in \Phi.$ \\
$(g)$ $\Phi$ is called an ordinary class if it is a subtractive lattice, $f, g\in \Phi$ implies $f.g\in \Phi$ and $f\in \Phi, f(x)\neq 0,$ for all $x\in X$ implies $\frac{1}{f}\in \Phi.$
 \end{defi}

\begin{defi}\cite{Papanastassiou 2002}
A sequence of functions $\{f_n \}_{n\in \mathbb{N}}$ in $\Phi$ is said to converge uniformly equally to a function $f$ in $\Phi$ (written as $f_n \xrightarrow{u.e} f$) if there exists a sequence $\{\varepsilon_n \}_{n\in \mathbb{N}}$ of positive reals converging to zero and a natural number $n_0$ such that the cardinality of the set $\{n\in \mathbb{N}: |f_n (x)-f(x)|\geq \varepsilon_n\}$ is at most $n_0$, for each $x\in X$. 
\end{defi}

\begin{defi}\cite{Papanastassiou 2002}
A sequence of functions $\{f_n \}_{n\in \mathbb{N}}$ in $\Phi$ is said to converge strongly uniformly equally to a function $f$ in $\Phi$ (written as $f_n \xrightarrow{s.u.e} f$) if there exists a sequence $\{\varepsilon_n \}_{n\in \mathbb{N}}$ of positive reals with $\Sigma_{n=1}^{\infty} \varepsilon _n < \infty$ and $n_0 \in \mathbb{N}$ such that
$|\{n\in \mathbb{N}: |f_n (x)-f(x)|\geq \varepsilon_n\}|\leq n_0$, for each $x\in X$.
\end{defi}

\begin{defi}\cite{Das Dutta Pal}
A sequence of functions $\{f_n \}_{n\in \mathbb{N}}$ is said to converge $\mathcal{I^*}$-uniformly equally to a function $f$  (written as $f_n \xrightarrow{\mathcal{I^*}\text{-}u.e.} f$) if there exist a sequence $\{\varepsilon_n \}_{n\in \mathbb{N}}$ of positive reals converging to zero, $M=M(\{\varepsilon_n \})\in \mathcal F(\mathcal{I})$ and $k(\{\varepsilon_n \})\in \mathbb{N}$ such that  $|\{n\in M : |f_n (x)-f(x)|\geq \varepsilon_n\}|$ is at most $k=k(\{\varepsilon_n \}$ for all $x\in X$. 
\end{defi}

\begin{defi}\cite{Das Dutta}
A sequence of functions $\{f_n \}_{n\in \mathbb{N}}$ is said to converge $\mathcal{I^*}$-strongly uniformly equally to a function $f$  (written as $f_n \xrightarrow{\mathcal{I^*}\text{-}s.u.e.} f$) if there exists a sequence $\{\varepsilon_n \}_{n\in \mathbb{N}}$ of positive reals with  $\Sigma_{n=1}^{\infty} \varepsilon _n < \infty$, a set $M=M(\{\varepsilon_n \})\in \mathcal F(\mathcal{I})$ and $k(\{\varepsilon_n \})\in \mathbb{N}$ such that  $|\{n\in M : |f_n (x)-f(x)|\geq \varepsilon_n\}|$ is at most $k=k(\{\varepsilon_n \}$ for all $x\in X$.
\end{defi}

Now we recall the definitions of $\alpha$-convergence, $\alpha$-uniform equal and $\alpha$-strong uniform equal convergence.

\begin{defi} (see \cite{Das Papanastassiou})
Let $(X, d)$ be a metric space and $f, f_n : X\rightarrow \mathbb{R}$, $n\in \mathbb{N}$. Then   $\{f_n \}_{n\in \mathbb{N}}$ $\alpha$-converges to $f$ (written as $ f_n\xrightarrow{\alpha}f$) if for any $x\in X$ and for any sequence $\{x_n\}_{n\in \mathbb{N}}$ of points of $X$ converging to $x$, $(f_n(x_n))$ converges to $f(x)$.
\end{defi}
In the literature, $\alpha$-convergence is also known as continuous convergence \cite{Stoilov}.

\begin{defi}\cite{Das Papanastassiou}
Let $(X, d)$ be a metric space and $f, f_n : X\rightarrow \mathbb{R}$, $n\in \mathbb{N}$. Then   $\{f_n \}_{n\in \mathbb{N}}$  is said to converge $\alpha$-uniformly equally to a function $f$  (written as $f_n \xrightarrow{\alpha \text{-}u.e} f$) if there exists a sequence $\{\varepsilon_n \}_{n\in \mathbb{N}}$ of positive reals converging to zero and a natural number $n_0$ such that $|\{n\in \mathbb{N}: |f_n (x_n)-f(x)|\geq \varepsilon_n\}| \leq n_0$ for each $x\in X$ and $x_n \rightarrow x$.
\end{defi}

\begin{defi}\cite{Das Papanastassiou}
Let $(X, d)$ be a metric space and $f, f_n : X\rightarrow \mathbb{R}$, $n\in \mathbb{N}$. Then   $\{f_n \}_{n\in \mathbb{N}}$  is said to converge $\alpha$-strongly uniformly equally to a function $f$  (written as $f_n \xrightarrow{\alpha \text{-}s.u.e} f$) if there exists a convergent series $\Sigma_{n=1}^{\infty}\varepsilon_n $ of positive reals  and a natural number $n_0$ such that $|\{n\in \mathbb{N}: |f_n (x_n)-f(x)|\geq \varepsilon_n\}| \leq n_0$ for every $x\in X$ and $x_n \rightarrow x$.
\end{defi}

\section{\textbf{$\mathcal{I^*}\text{-}\alpha\text{-}$uniform equal  and $\mathcal{I^*}\text{-}\alpha\text{-}$strong uniform equal convergence }}
Throughout the paper $X$ stands for a metric space and $f, f_n$ are real valued functions defined on $X$, $\Phi ^{\mathcal{I^*}\text{-}\alpha\text{-}u.e.}$ and $\Phi ^{\mathcal{I^*}\text{-}\alpha \text{-}s.u.e.}$  stand for the classes of all functions on $X$ which are $\mathcal{I^*}\text{-}\alpha$-uniform equal limits and $\mathcal{I^*}\text{-}\alpha$-strong uniform equal limits of sequences of functions respectively obtained from a class of functions $\Phi$. Then we first introduce the following definitions.

\begin{defi}
 A sequence $\{f_n\}_{n\in\mathbb{N}}$ is said to converge $\mathcal{I^*}\text{-}\alpha$-uniformly equally to $f$ (written as $f_n \xrightarrow{\mathcal{I^*}\text{-}\alpha \text{-}u.e.} f$) if there exists a sequence $(\varepsilon _n)_{n\in\mathbb{N}}$ of positive reals converging to zero, a set $M\in \mathcal{F(I)}$ and $n_0\in\mathbb{N}$ such that $|\{n\in M : |f_n (x_n)-f(x)|\geq \varepsilon _n \}|\leq n_0$  for each $x\in X$ and $(x_n)_{n\in M} \rightarrow x$.
\end{defi}

\begin{defi}
 A sequence $\{f_n\}_{n\in\mathbb{N}}$ is said to converge $\mathcal{I^*}\text{-}\alpha$-strongly uniformly equally to $f$ (written as $f_n \xrightarrow{\mathcal{I^*}\text{-}\alpha \text{-}s.u.e.} f$) if there exists a convergent series $\Sigma_{n=1}^{\infty}\varepsilon_n $ of positive reals,  a set $M\in \mathcal{F(I)}$ and $n_0\in\mathbb{N}$ such that $|\{n\in M : |f_n (x_n)-f(x)|\geq \varepsilon _n \}|\leq n_0$  for every $x\in X$ and $(x_n)_{n\in M} \rightarrow x$.
\end{defi}

\begin{rem}
From the definition we see that $f_n \xrightarrow{\mathcal{I^*}\text{-}\alpha \text{-}s.u.e.} f$ implies $f_n \xrightarrow{\mathcal{I^*}\text{-}\alpha \text{-}u.e.} f$. But the converse is not true as shown in  the following example. 
\end{rem}

\begin{exmp}
Let $\mathcal{I}$ be a non trivial proper admissible ideal of $\mathbb{N}$. Then we have an infinite set $B\in \mathcal{F}(\mathcal{I})$. Let $f, f_n : \mathbb{R}\rightarrow \mathbb{R}, n\in \mathbb{N}$ be functions defined by $f_n(x)=\begin{cases}
\frac{1}{n}\ if\ n\in B\\
0 \ if\ n\notin B
\end{cases}$ for each $x\in \mathbb{R}$ and $f(x)=0$ for all $x\in \mathbb{R}$. Suppose $\varepsilon _n = \frac{1}{\sqrt{n}}$. Then $\varepsilon _n \rightarrow 0$, as $n\rightarrow \infty$.  Now we have $|\{n\in B: |f_n (x_n)-f(x)|\geq \varepsilon _n\}|=|\{n\in B: \frac{1}{n}\geq \frac{1}{\sqrt{n}}\}|$ is at most one. Therefore $f_n \xrightarrow{\mathcal{I^*}\text{-}\alpha \text{-}u.e.} f$. But it is not $\mathcal{I^*}\text{-}\alpha$-strong uniform equal convergent. If possible let  $f_n \xrightarrow{\mathcal{I^*}\text{-}\alpha \text{-}s.u.e.} f$. Then there exists a convergent series $\Sigma_{n=1}^{\infty}\varepsilon_n $ of positive reals  and $n_0\in\mathbb{N}$ such that $|\{n\in B: |f_n (x_n)-f(x)|\geq \varepsilon _n \}|\leq n_0$ i.e. $|\{n\in \mathbb{N}: \frac{1}{n}\geq \varepsilon _n \}|\leq n_0$. Then there is $m\in\mathbb{N}$ such that $ \frac{1}{n}< \varepsilon _n$ for all $n>m$, which is a contradiction as $\Sigma \frac{1}{n}$ is divergent. Therefore $\{f_n\}_{n\in\mathbb{N}}$ does not converge $\mathcal{I^*}\text{-}\alpha$-strongly uniformly equally to the function $f$. 
\end{exmp}

We now observe the following equivalent conditions for the $\mathcal{I^*}\text{-}\alpha \text{-}u.e.$ and $\mathcal{I^*}\text{-}\alpha \text{-}s.u.e.$ convergences.

\begin{prop}
Let $f, f_n : X\rightarrow \mathbb{R}, n\in \mathbb{N}$. If $(\varepsilon _n)_{n\in \mathbb{N}}$ and $(\lambda _n)_{n\in \mathbb{N}}$ are two zero sequences of positive reals such that $0<\varepsilon _n \leq \lambda _n$ for every $n\in \mathbb{N}$ and $(\varepsilon _n)_{n\in \mathbb{N}}$ witnesses the $\mathcal{I^*}\text{-}\alpha \text{-}u.e.$ convergence then $(\lambda _n)_{n\in \mathbb{N}}$ also witnesses the same.
\end{prop}

\begin{proof}
By the given condition there exists a set $M\in \mathcal{F(I)}$ and $n_0\in\mathbb{N}$ such that $|\{n\in M : |f_n (x_n)-f(x)|\geq \varepsilon _n \}|\leq n_0$  for each $x\in X$ and $(x_n)_{n\in M} \rightarrow x$. Since  $\{n\in M : |f_n (x_n)-f(x)|\geq \lambda _n \}\subset \{n\in M : |f_n (x_n)-f(x)|\geq \varepsilon _n \}$, therefore $|\{n\in M : |f_n (x_n)-f(x)|\geq \lambda _n \}|\leq n_0$  for each $x\in X$ and $(x_n)_{n\in M} \rightarrow x$. 
\end{proof}

\begin{rem}
If  $\Sigma_{n=1}^{\infty}\varepsilon_n $ and $\Sigma_{n=1}^{\infty}\lambda_n $ are two convergent series of positive reals such that $0<\varepsilon _n \leq \lambda _n$ for every $n\in \mathbb{N}$ then, by similar arguments, the above result also holds for $\mathcal{I^*}$-$\alpha$-strong uniform equal convergences.
\end{rem}

\begin{lem}\label{lem1}
Let $ f_n : X\rightarrow \mathbb{R}, n\in \mathbb{N}$. If $f_n\xrightarrow{\mathcal{I^*}\text{-}\alpha \text{-} u.e.} 0$ then $f_{n}^2 \xrightarrow{\mathcal{I^*}\text{-}\alpha \text{-} u.e.} 0$.
\end{lem}

 \begin{proof}
By the condition, there exists a sequence $(\varepsilon _n)_{n\in \mathbb{N}}$ of positive reals converging to zero, a set $M\in \mathcal{F(I)}$ and $n_0\in\mathbb{N}$ such that $|\{n\in M : |f_n (x_n)|\geq \varepsilon _n \}|\leq n_0$  for each $x\in X$ and $(x_n)_{n\in M} \rightarrow x$.  Therefore $|\{n\in M : |f_{n}^2 (x_n)|\geq \varepsilon _{n}^2 \}|\leq |\{n\in M : |f_n (x_n)|\geq \varepsilon _n \} |\leq n_0$ for each $x\in X$ and $(x_n)_{n\in M} \rightarrow x$ and hence $f_{n}^2 \xrightarrow{\mathcal{I^*}\text{-}\alpha \text{-}u.e.} 0$.
 \end{proof}

\begin{lem}\label{lem2}
Let $f, f_n : X\rightarrow \mathbb{R}, n\in \mathbb{N}$. If $f$ is a non zero constant function and $f_n\xrightarrow{\mathcal{I^*}\text{-}\alpha \text{-} u.e.} f$ then $f_n .f\xrightarrow{\mathcal{I^*}\text{-}\alpha \text{-} u.e.} f^2$.
\end{lem}

\begin{proof}
 Let $f(x) = c$ for each $x\in X$ where $c$ is a non zero constant. Since $f_n\xrightarrow{\mathcal{I^*}\text{-}\alpha \text{-} u.e.} f$ then there exist a sequence $(\varepsilon _n)_{n\in \mathbb{N}}$ of positive reals converging to zero,  $M\in \mathcal{F(I)}$ and  $n_0\in \mathbb{N}$ such that $|\{n\in M : |f_n (x_n)-f(x)|\geq \varepsilon _n \}|\leq n_0$ for each $x\in X$ and $(x_n)_{n\in M} \rightarrow x$. 
Therefore $|\{n\in M: |(f_n .f) (x_n)-(f.f)(x)|\geq |c| \varepsilon _n \}|= |\{n\in M : |f_n (x_n)-f(x)|\geq \varepsilon _n \}|\leq n_0$ for each $x\in X$ and $(x_n)_{n\in M} \rightarrow x$. It follows $f_n .f\xrightarrow{\mathcal{I^*}\text{-}\alpha \text{-} u.e.} f^2$.
\end{proof}

\begin{note}
By similar techniques we can prove that if $f$ is bounded and $f_n\xrightarrow{\mathcal{I^*}\text{-}\alpha \text{-} u.e.} f$ then $f_n .f\xrightarrow{\mathcal{I^*}\text{-}\alpha \text{-} u.e.} f^2$ and also $f_n ^2\xrightarrow{\mathcal{I^*}\text{-}\alpha \text{-} u.e.} f^2$.
\end{note}

\begin{thm}\label{thm3.1}
If  $f_n\xrightarrow{\mathcal{I^*}\text{-}\alpha \text{-} u.e.} f$,  $g_n\xrightarrow{\mathcal{I^*}\text{-}\alpha \text{-} u.e.} g$ and $a,b \in \mathbb{R}$, then  $a f_n + b g_n\xrightarrow{\mathcal{I^*}\text{-}\alpha \text{-} u.e.} af+bg$.
\end{thm}

\begin{proof}
By the condition there exist sequences $(\sigma _n)_{n\in \mathbb{N}}$ and $(\rho _n)_{n\in N}$ of positive reals both converging to zero, $M_1, M_2 \in \mathcal{F}(\mathcal{I})$ and $n_1 , n_2 \in \mathbb{N}$ such that $|\{n\in M_1 : |a f_n (x_n)-a f(x)|\geq |a| \sigma _n \}|\leq n_1$ for each $x\in X$ and $(x_n)_{M_1} \rightarrow x$ and $|\{n\in M_2 : |b g_n (x_n)-b g(x)|\geq |b| \rho _n \}|\leq n_2$ for each $x\in X$ and $(x_n)_{M_2} \rightarrow x$. Let $\varepsilon _n = a\sigma _n + b\rho _n$. Also let $B_1=\{n\in M_1 : |a f_n (x_n)-a f(x)|\geq |a| \sigma _n \}$, $B_2=\{n\in M_2 : |b g_n (x_n)-b g(x)|\geq |b| \rho _n \}$ and $B=\{n\in M_1\cap M_2 : |(a f_n + b g_n) (x_n)-(a f+b g)(x)|\geq  \varepsilon _n \}$. Now for $n\in B$, we have  $\varepsilon_n\leq |(a f_n + b g_n) (x_n)-(a f+b g)(x)|\leq |a f_n (x_n)-a f(x)|+|b g_n (x_n)-b g(x)|$. So $B\subset \{n\in\mathbb{N}:|a f_n (x_n)-a f(x)|\geq \varepsilon_n\}\cup \{n\in \mathbb{N}:  |b g_n (x_n)-b g(x)|\geq \varepsilon_n\}\subset B_1\cup B_2$.  Therefore $|\{n\in M_1\cap M_2 : |(a f_n + b g_n) (x_n)-(a f+b g)(x)|\geq  \varepsilon _n \}|\leq n_1+n_2$ for each $x\in X$ and $(x_n)_{M_1\cap M_2} \rightarrow x$, where $M_1\cap M_2\in \mathcal{F(I)}$ and $\lim_{n\to \infty}\varepsilon_n=0$. Hence the theorem follows.
\end{proof}

\begin{cor}\label{cor1}
If  $f_n\xrightarrow{\mathcal{I^*}\text{-}\alpha \text{-} u.e.} f$,  $g_n\xrightarrow{\mathcal{I^*}\text{-}\alpha \text{-} u.e.} g$ then $f_n+g_n\xrightarrow{\mathcal{I^*}\text{-}\alpha \text{-} u.e.} f+g$ and $f_n-g_n\xrightarrow{\mathcal{I^*}\text{-}\alpha \text{-} u.e.} f-g$.
\end{cor}

\begin{lem}\label{lem3}
 If $f$, $g$ are bounded and $f_n\xrightarrow{\mathcal{I^*}\text{-}\alpha \text{-} u.e.} f$,  $g_n\xrightarrow{\mathcal{I^*}\text{-}\alpha \text{-} u.e.} g$ then $f_n . g_n\xrightarrow{\mathcal{I^*}\text{-}\alpha \text{-} u.e.} f.g$. 
\end{lem}

\begin{proof}
Applying the Lemma \ref{lem1}, Lemma \ref{lem2} and Corollary \ref{cor1},  we have
$f_n . g_n = \frac{(f_n + g_n)^2 -(f_n + g_n)^2}{4}$
$\xrightarrow{\mathcal{I^*}\text{-}\alpha \text{-}u.e.}\frac{(f + g)^2 -(f + g)^2}{4} = f.g$.
\end{proof}

\begin{rem}
The Lemma \ref{lem1}, Lemma \ref{lem2}, Theorem \ref{thm3.1} and Lemma \ref{lem3} also hold good for $\mathcal{I^*}\text{-}\alpha \text{-}$strong uniform equal convergences.
\end{rem}

\begin{thm}
Let $\Phi$ be a class of functions on $X.$ If $\Phi$ is a lattice, a translation lattice, a congruence lattice, a weakly affine lattice, an affine lattice or a subtractive lattice then so is $\Phi ^{\mathcal{I^*}\text{-}\alpha \text{-} u.e.}$.
\end{thm}

\begin{proof}
 Let $\Phi$ be a lattice. Then $\Phi$ contains all constant functions. If $f$ is a constant function  and $f_n =f, n\in \mathbb{N}$, belonging to $\Phi$, then for any sequence $(\varepsilon _n)_{n\in \mathbb{N}}$ of positive reals converging to zero, $M\in \mathcal{F}(\mathcal{I})$ the set $\{n\in \mathbb{N} : |f_n (x_n)-f(x)|\geq \varepsilon _n \}$ is empty for each $x\in X$ and $(x_n)_{n\in M} \rightarrow x$ and so $\Phi^{\mathcal{I^*}\text{-}\alpha \text{-} u.e.}$ contains all constants functions. Now we approach to show $max(f,g)$ and $min(f,g)\in \Phi^{ \mathcal{I^*}\text{-}\alpha \text{-} u.e.} $. For, if $f_n\xrightarrow{\mathcal{I^*}\text{-}\alpha \text{-} u.e.} f$, then then there exists a sequence $(\varepsilon _n)_{n\in \mathbb{N}}$ of positive reals converging to zero,  $M\in \mathcal{F(I)}$ and  $n_0\in \mathbb{N}$ such that $|\{n\in M : |f_n (x_n)-f(x)|\geq \varepsilon _n \}|\leq n_0$ for each $x\in X$ and $(x_n)_{n\in M} \rightarrow x$. Since $\Big||f_n(x_n )|-|f(x)|\Big|\leq \big|f_n (x_n )-f(x)\big|$ and observe that $\{n\in M : \Big||f_n(x_n )|-|f(x)|\Big|\geq \varepsilon _n \}\subset \{n\in M : \big|f_n(x_n )-f(x)\big|\geq \varepsilon _n \}$. Therefore $|\{n\in M : \Big||f_n(x_n )|-|f(x)|\Big|\geq \varepsilon _n \}|\leq n_0$  for each $x\in X$ and $(x_n)_{n\in M} \rightarrow x.$ 
Hence $|f|\in \Phi ^{\mathcal{I^*}\text{-}\alpha \text{-} u.e.}$. 

Now using the Corollary \ref{cor1} we have   $\frac{f_n + g_n}{2} + \frac{|f_n - g_n|}{2}\xrightarrow{\mathcal{I^*}\text{-}\alpha \text{-} u.e.}\frac{f + g}{2} + \frac{|f - g|}{2}$
which proves $max(f, g)\in \Phi ^{\mathcal{I^*}\text{-}\alpha \text{-} u.e.}, $ in the similar way we can  show that  $min(f, g)\in \Phi ^{\mathcal{I^*}\text{-}\alpha \text{-} u.e.}$. Therefore $\Phi ^{\mathcal{I^*}\text{-}\alpha \text{-} u.e.}$ is a lattice. The proof of the remaining parts stay on the similar arguments.
\end{proof}

\begin{thm}
Let $\Phi$ be a class of functions on $X$ and  $f\in \Phi^{\mathcal{I^*}\text{-}\alpha \text{-} u.e}$ be bounded on $X$ and $f(x)\neq 0$ for each $x\in X$. If $\frac{1}{f}$ is bounded on $X$, then $\frac{1}{f}\in  \Phi ^{\mathcal{I^*}\text{-}\alpha \text{-} u.e}$.
\end{thm}

\begin{proof}
 Since $f\in \Phi^{\mathcal{I^*-}\alpha-u.e},$ there exists a sequence $(\sigma _n)_{n\in\mathbb{N}}$ of positive reals converging to zero,  a set $M\in \mathcal{F(I)}$ and  $n_0\in \mathbb{N}$ such that $|\{n\in M : |f_n (x_n)-f(x)|\geq \sigma _n \}|\leq n_0$  for each $x\in X$ and $(x_n)_{n\in M} \rightarrow x$. 
Let $A=\{n\in M : |f_n (x_n)-f(x)|\geq \sigma _n \}$.
Now $\frac{1}{f}$ being bounded on $X$, there exists a positive number $\lambda$ such that $\frac{1}{|f(x)|}< \lambda$ for all $x\in X$. 
Define $g_n\in \Phi$ by $g_n(x)=max\{f_n(x),\sqrt{\sigma_n}\}$ for $x\in X$ and 
$n\in \mathbb{N}$. 
Now  we have 
\begin{align*}
 &\{n\in M : |g_n(x_n)-f(x)|\geq \sigma _n \}    \\
\subset \ &\{n\in M :g_n = f_n, |g_n(x_n)-f(x)|\geq \sigma _n \}\cup \{n\in M : g_n =\sqrt{\sigma_n},  |g_n(x_n)-f(x)|\geq \sigma _n \} \\
\subseteq \ & A \cup \{n\in M : g_n =\sqrt{\sigma_n},  g_n(x_n)-f(x)\geq \sigma _n \}\cup \{n\in M : g_n =\sqrt{\sigma_n},  -g_n(x_n)+f(x)\geq \sigma _n \} \\
\subseteq \ & A \cup \{n\in M : f(x)\leq \sqrt{\sigma_n}-\sigma _n \}\cup  \{n\in M : f(x)\geq f_n(x_n)+\sigma _n \}  \\
\subseteq \ & A \cup \{n\in M : f(x)\leq \sqrt{\sigma_n}\}\cup A \\
\subseteq \ & A \cup  \{n\in M : f^2 (x)\leq \sigma_n\}.       
\end{align*}
Since $\frac{1}{|f(x)|}< \lambda$ for all $x\in X$ and $(\sigma _n)_{n\in \mathbb{N}}$ is convergent, $\{n\in M : f^2 (x)\leq \sigma_n\}$ is finite.  
Let $|\{n\in M : f^2 (x)\leq \sigma_n\}|\leq n_1$.
Therefore $|\{n\in M : |g_n(x_n)-f(x)|\geq \sigma _n \}| \leq n_0 + n_1$ for  $x\in X$ and $(x_n)_{n\in M} \rightarrow x$.
Now $|\{n\in M : |\frac{1}{g_n (x_n)} - \frac{1}{f(x)}|\geq \lambda \sqrt{\sigma _n} \}|   
 =  |\{n\in M : \frac{|g_n (x_n)-f(x)|}{|g_n (x_n)||f(x)|} \geq  \lambda \sqrt{\sigma _n} \}|
 \leq  |\{n\in M : |g_n(x_n)-f(x)|\geq \sigma _n \}|
 \leq  n_0 + n_1 $
 for  $x\in X$ and $(x_n)_{n\in M} \rightarrow x.$
As $\lambda \sqrt{\sigma _n} \rightarrow 0$, so   $ \frac{1}{g_n}  \xrightarrow{\mathcal{I^*}\text{-}\alpha \text{-} u.e.}\frac{1}{f}$. 
Hence $\frac{1}{f}\in  \Phi ^{\mathcal{I^*}\text{-}\alpha \text{-} u.e}$. This proves the theorem.
\end{proof}

\begin{thm}
Let $\Phi$ be a class of functions on $X.$ If $\Phi$ is a lattice, a translation lattice, a congruence lattice, a weakly affine lattice, an affine lattice or a subtractive lattice then so is $\Phi ^{\mathcal{I^*}\text{-}\alpha \text{-} s.u.e.}$.
\end{thm}

\begin{proof}
Let $\Phi$ be a lattice. Then $\Phi$ contains all constant functions. We can  easily prove that  $\Phi ^{\mathcal{I^*}\text{-}\alpha \text{-} s.u.e.}$ also contains the constant functions. Now we show if  $f_n\xrightarrow{\mathcal{I^*}\text{-}\alpha \text{-} s.u.e.} f$ then $|f_n|\xrightarrow{\mathcal{I^*}\text{-}\alpha \text{-} s.u.e.} |f|$. 
For, let there  exist a convergent series $\Sigma _{n=1}^{\infty}\varepsilon _n$ of positive reals,  a set $M\in \mathcal{F(I)}$ and  $n_0\in \mathbb{N}$ such that $|\{n\in M : |f_n (x_n)-f(x)|\geq \varepsilon _n \}|\leq n_0$  for each $x\in X$ and $(x_n)_{n\in M} \rightarrow x$.
Since for any two reals $a_1$ and $a_2$,  $\Big||a_1|-|a_2|\Big|\leq \big|a_1-a_2\big|$, it follows that 
$\{n\in M : \Big||f_n(x_n )|-|f(x)|\Big|\geq \varepsilon _n \}\subset \{n\in M : |f_n(x_n )-f(x)|\geq \varepsilon _n \}$. Therefore $|\{n\in M: \Big||f_n(x_n )|-|f(x)|\Big|\geq \varepsilon _n \}|\leq n_0$  for each $x\in X$ and $(x_n)_{n\in M} \rightarrow x$.
Hence $|f|\in \Phi ^{\mathcal{I^*}\text{-}\alpha \text{-} s.u.e.}$.
Now if $f_n\xrightarrow{\mathcal{I^*}\text{-}\alpha \text{-} s.u.e.} f$,  $g_n\xrightarrow{\mathcal{I^*}\text{-}\alpha  \text{-} s.u.e.} g$ and $a, b \in \mathbb{R}$ then there exist convergent series $\Sigma _{n=1} ^{\infty }\sigma _n$ and $\Sigma _{n=1} ^{\infty }\eta _n$ of positive reals, $M_1, M_2\in \mathcal{F(I)}$  and $n_1, n_2 \in \mathbb{N}$  such that  $|\{n\in M_1 : |a f_n (x_n)-a f(x)|\geq \sigma _n \}|\leq n_1$ for each $x\in X$ and $(x_n)_{M_1} \rightarrow x$
 $|\{n\in M_2 : |b g_n (x_n)-b g(x)|\geq \eta _n \}|\leq n_2$ for each $x\in X$ and $(x_n)_{M_2} \rightarrow x$.
Suppose that $\varepsilon_n =  max\{2|a| \sigma _n,2|b| \eta _n \}$ and $n_0 = n_1 + n_2$. Now following the similar techniques of the Theorem \ref{thm3.1} we have
 $|\{n\in M_1\cap M_2 : |(a f_n+b g_n) (x_n)-(a f +b g)(x)|\geq \varepsilon _n \}|\leq n_0$ for each $x\in X$ and $(x_n)_{M_1\cap M_2} \rightarrow x$.
Now since $M_1\cap M_2\in \mathcal{F(I)}$ and 
$\Sigma _{n=1}^{\infty}\epsilon _n = \Sigma _{n=1}^{\infty} max\{2|a|\sigma _n,2|b| \eta _n \}
\leq  \Sigma _{n=1}^{\infty} \{2|a|\sigma _n +2|b| \eta _n \}=2|a|\Sigma _{n=1}^{\infty}\sigma _n + 2|b|\Sigma _{n=1}^{\infty}\eta _n
< \infty$.
Hence $a f_n +b g_n \xrightarrow{\mathcal{I^*}\text{-}\alpha \text{-}    s.u.e.}a f + b g$. Therefore
  $\frac{f_n + g_n}{2} + \frac{|f_n - g_n|}{2}\xrightarrow{\mathcal{I^*}\text{-}\alpha \text{-} s.u.e.}\frac{f + g}{2} + \frac{|f - g|}{2}$
which gives $max(f, g)\in \Phi ^{\mathcal{I^*}\text{-}\alpha \text{-} s.u.e.} $. In the similar way we can easily show that  $min(f, g)\in \Phi ^{\mathcal{I^*}\text{-}\alpha \text{-}  s.u.e.} $. Therefore $\Phi$ is a lattice.The proofs of the remaining parts can be shown by the similar arguments.
\end{proof}

\section{\textbf{ Some types of $\mathcal{I}$-convergence preserved under uniform conjugacy}}

First we recall the following preliminary concept found in classical mathematical theory for the sake of completeness.

Let $(X,d)$ and $(Y,\rho)$ be two metric spaces. If $h:X\rightarrow Y$ is a homeomorphism such that $h$ is uniformly continuous on $X$ and $h^{-1}$ is uniformly continuous on $Y$, then $h$ is said to be uniform homeomorphism.

Now we recall the definitions of uniform conjugacy and $\mathcal{I}$-exhaustiveness for sequences of functions which will be useful in the sequel.
\begin{defi}\cite{Tian Chen}
Let $(X,d)$ and $(Y,\rho)$ be two metric spaces, $F=\{f_n\}_{n\in \mathbb{N}}$ and $G=\{g_n\}_{n\in \mathbb{N}}$ be two sequences of maps in $X$ and $Y$ respectively and $h:X\rightarrow Y$ be a homeomorphism. If for any $k\in \{1,2,\ldots \}$, $g_k (h(x))=h(f_k(x))$ for every $x\in X$, then $F$ and $G$ are said to be $h$-conjugate. In particular, if $h$ is a uniform homeomorphism then $F$ and $G$ are said to be uniformly $h$-conjugate.
\end{defi}

\begin{defi}\cite{Papachristodoulos Papanastassiou Wilczynski}
Let $(X,d)$ and $(Y,\rho)$ be two metric spaces. We suppose that $\{f_n\}_{n\in \mathbb{N}}$ is a sequence of functions from $X$ to $Y$. Then $\{f_n\}_{n\in \mathbb{N}}$ is said to be $\mathcal{I}$-exhaustive at $x$ iff for every $\varepsilon >0$, there exist a $\delta >0$, a set $A\in \mathcal{I}$ (depending on $\varepsilon$ and $x$) such that $\rho (f_n (x), f_n (y))<\varepsilon$ whenever $n\in \mathbb{N}\setminus A$ and for all $y\in S(x, \delta)$ where $S(x, \delta)=\{y\in X : d(x,y)<\delta \}$.The sequence $\{f_n\}_{n\in \mathbb{N}}$ is said to be  $\mathcal{I}$-exhaustive on $X$ iff it is $\mathcal{I}$-exhaustive at every $x\in X$.
\end{defi}

\begin{defi}\cite{Papachristodoulos Papanastassiou Wilczynski}
Let $(X,d)$ and $(Y,\rho)$ be two metric spaces and $f, f_n : X\rightarrow Y$. We say that the sequence $\{f_n\}_{n\in \mathbb{N}}$ $\mathcal{I}$-$\alpha$-converges to $f$ (written as $f_n\xrightarrow{\mathcal{I}\text{-}\alpha} f$) at $x_0\in X$ if and only if for each sequence $\{x_n\}_{n\in \mathbb{N}}$ of points of $X$ if $x_n\xrightarrow{\mathcal{I}}x_0)$, then $f_n(x_n)\xrightarrow{\mathcal{I}}f(x_0)$.
\end{defi}

Now we prove that $\mathcal{I}$-exhaustiveness, $\mathcal{I}$-uniform convergence and $\mathcal{I}$-$\alpha$-convergence of sequences of functions are preserved under uniform conjugacy.
\begin{thm}
Let $(X,d)$ and $(Y,\rho)$ be two metric spaces   and $h:X\rightarrow Y$ be a uniform homeomorphism and let  $F=\{f_n\}_{n\in \mathbb{N}}$ and $G=\{g_n\}_{n\in \mathbb{N}}$ be uniformly $h$-conjugate. If $F$ is $\mathcal{I}$-exhaustive on $X$ then $G$ is also $\mathcal{I}$-exhaustive on $Y$.
\end{thm}

\begin{proof}
Since  $F$ is $\mathcal{I}$-exhaustive on $ X$, it is $\mathcal{I}$-exhaustive at every $x\in X$. Let $y\in Y$. Now $h$ being onto, there is $x\in X$ such that $h(x)=y$. First we show $G$ is $\mathcal{I}$-exhaustive at $y$. Let $\varepsilon >0$ be given. Since $h$ is uniformly continuous, there exists  $\delta>0$ such that for every $x_1, x_2 \in X$, $d(x_1, x_2)<\delta \Rightarrow \rho(h(x_1), h(x_2))<\varepsilon$. Since $F$ is $\mathcal{I}$-exhaustive at $x\in X$, there exists $\beta>0$ and a set $A\in \mathcal{I}$ such that for $n\in A^{c}$, $d(f_n(x),f_n(z))<\delta$ for all $z\in S(x,\beta)$ i.e. for all $z$ satisfying $d(x,z)<\beta$ implies $d(f_n(x),f_n(z))<\delta$. Since $h^{-1}$ is uniformly continuous, there exists a $\eta >0$ such that for every $y_1, y_2 \in Y$, $\rho(y_1, y_2)<\eta \Rightarrow d(h^{-1}(y_1), h^{-1}(y_2))<\beta$. Let $b\in S(y, \eta)$ then $\rho(y, b)< \eta$, which gives $d(h^{-1}(y), h^{-1}(b))< \beta$. It implies $d(f_n(x),f_n(z))<\delta$ for $n\in A^{c}$ where $z=h^{-1}(b)$. Therefore finally we get $\rho(h(f_n(x)),h(f_n(z)))<\varepsilon$ for  $n\in A^{c}$. Since $F$ and $G$ are  $h$-conjugate, we have $\rho(g_n(h(x)), g_n(h(z)))<\varepsilon$ i.e. $\rho(g_n(y), g_n(z))<\varepsilon$ for all $n\in A^{c}$. Hence  $G$ is $\mathcal{I}$-exhaustive at $y$. Since $y\in Y$ is arbitrary, $G$ is $\mathcal{I}$-exhaustive on $Y$.
\end{proof}

\begin{rem}
In the similar way we can prove that if $G$ is $\mathcal{I}$-exhaustive on $Y$ then $F$ is $\mathcal{I}$-exhaustive on $X$.
\end{rem}

\begin{thm}
Let $(X,d)$ and $(Y,\rho)$ be two metric spaces   and $h:X\rightarrow Y$ be a uniform homeomorphism and let  $F=\{f_n\}_{n\in \mathbb{N}}$ and $G=\{g_n\}_{n\in \mathbb{N}}$ be uniformly $h$-conjugate. Then $F$ is $\mathcal{I}$-uniformly convergent iff $G$ is so.
\end{thm}

\begin{proof}
Let $\{f_n\}_{n\in \mathbb{N}}$ be $\mathcal{I}$-uniformly convergent to $f$ where $f:X\rightarrow X$. We show that $\{g_n\}_{n\in \mathbb{N}}$ is $\mathcal{I}$-uniformly convergent to $g$ where $g:X\rightarrow X$ such that $f$ and $g$ are uniformly $h$-conjugate. Let $\varepsilon >0$ be given.  Since $h$ is uniformly continuous, there exists a $\delta >0$ such that for every $x_1, x_2 \in X$, $d(x_1, x_2)< \delta \Rightarrow \rho(h(x_1),h(x_2))<\varepsilon$. Since $\{f_n\}_{n\in \mathbb{N}}$ is $\mathcal{I}$-uniformly convergent to $f$, there is a set $A\in \mathcal{I}$ such that for $n\in A^{c}$, $d(f_n(x),f(x))<\delta$ for all $x\in X$, which gives $\rho(h(f_n(x)),h(f(x)))<\varepsilon$. Now by the condition $\rho(g_n(h(x)), g(h(x)))<\varepsilon$ for all $n\in A^{c}$ and for every $x\in X$. Since $h$ is bijective, we have $\rho(g_n(y),g(y))<\varepsilon$ for all $n\in A^{c}$ and for every $x\in X$. Hence  $\{g_n\}_{n\in \mathbb{N}}$ is $\mathcal{I}$-uniformly convergent to $g$. We can easily prove the converse part.
\end{proof}

\begin{prop}
Let $F$ and $G$ be uniformly $h$-conjugate. If $\{f_n\}_{n\in \mathbb{N}}$ converges $\mathcal{I}$-pointwise to $f$ then $\{g_n\}_{n\in \mathbb{N}}$ converges $\mathcal{I}$-pointwise to $g$ where $f$ and $g$ are $h$-conjugate.
\end{prop}

\begin{note}
It was shown in \cite{Balaz} that if $f: X\subset \mathbb{R}\rightarrow \mathbb{R}$ then $C(\mathcal{I})=C(\mathcal{I}_f)$ where $C(\mathcal{I}):$ the class of all functions $\mathcal{I}$-continuous on $X$ and $C(\mathcal{I}_f):$ the class of all functions continuous on $X$ in usual sense.
\end{note}

\begin{thm}
Let $(X,d)$ and $(Y,\rho)$ be two metric spaces   and $h:X\rightarrow Y$ be a uniform homeomorphism and let  $F=\{f_n\}_{n\in \mathbb{N}}$ and $G=\{g_n\}_{n\in \mathbb{N}}$ be uniformly $h$-conjugate. If $\{f_n\}_{n\in \mathbb{N}}$ $\mathcal{I}\text{-}\alpha$-converges to $f:X\rightarrow X$ then $\{g_n\}_{n\in \mathbb{N}}$ $\mathcal{I}\text{-}\alpha$-converges to $g:Y\rightarrow Y$ where $f$ and $g$ are $h$-conjugate.
\end{thm}

\begin{proof}
Let $y\in Y$ and $y_n\xrightarrow{\mathcal{I}}y$. Since  $h$ is onto, we get $x,x_n\in X$, $n\in \mathbb{N}$, such that $h(x_n)=y_n$ and $h(x)=y$. So $y_n\xrightarrow{\mathcal{I}}y \Rightarrow h(x_n)\xrightarrow{\mathcal{I}}h(x)$.  Since  $h^{-1}$ is continuous,  $x_n\xrightarrow{\mathcal{I}}x$. Then by the condition we have $f_n(x_n)\xrightarrow{\mathcal{I}}f(x)$. Now $h$ being continuous,  $h(f_n(x_n))\xrightarrow{\mathcal{I}}h(f(x))$. Finally by the given condition we have $g_n(h(x_n))\xrightarrow{\mathcal{I}}g(h(x))$ i.e. $g_n(y_n)\xrightarrow{\mathcal{I}}g(y)$. This proves the theorem.
\end{proof}

\textbf{Acknowledgement.}
The second author is grateful to The Council of  Scientific  and Industrial Research (CSIR), HRDG, India, for the grant of Junior Research Fellowship during the preparation of this paper.

\end{document}